\documentclass{amsproc}

\usepackage {amssymb}
\usepackage{verbatim}
\usepackage{graphicx,color}

\newcommand{\Ind}{\mathrm{Ind}}
\newcommand{\alg}{\mathrm{alg}}
\newcommand{\Hom}{\mathrm{Hom}}

\newtheorem{theorem}{Theorem}[section]

\newtheorem{prop}[theorem]{Proposition}
\newtheorem{cor}[theorem]{Corollary}

\theoremstyle{definition}

\newtheorem{example}[theorem]{Example}

\theoremstyle{remark}

\numberwithin{equation}{section}

\renewcommand{\top}{{\mathrm{top}}}
\newcommand{\ol}{\overline}

\newcommand{\field}{\mathbb}

\newcommand{\C}{{\field C}}
\newcommand{\R}{{\field R}}

\newcommand{\lam}{\lambda}

\newcommand{\AV}{\mathrm{AV}}

\newcommand{\CC}{\mathrm{CC}}
\renewcommand{\H}{\mathrm{H}}

\newcommand{\triv}{1 \! \! 1}

\newcommand{\sgn}{\mathrm{sgn}}
\newcommand{\lra}{\longrightarrow}

\newcommand{\wt}{\widetilde}

\newcommand{\bs}{\backslash}

\renewcommand{\-}{\! - \!}

\newcommand{\beq}{\begin{equation}}
\newcommand{\eeq}{\end{equation}}

\newcommand{\U}{\mathrm{U}}

\newcommand{\Sp}{\mathrm{Sp}}
\newcommand{\st}{\mathrm{st}}

\newcommand{\caSP}{\mathcal{S}\mathcal{P}}

\newcommand{\SP}{\caSP}
\newcommand{\SO}{\mathrm{SO}}

\newcommand\vG{^\vee\!G}

\newcommand{\frg}{\mathfrak{g}}
\newcommand{\frh}{\mathfrak{h}}

\newcommand{\frk}{\mathfrak{k}}
\newcommand{\frl}{\mathfrak{l}}
\newcommand{\frn}{\mathfrak{n}}

\newcommand{\frp}{\mathfrak{p}}

\newcommand{\frs}{\mathfrak{s}}

\newcommand{\bbC}{\mathbb{C}}

\newcommand{\bbR}{\mathbb{R}}

\newcommand{\bbZ}{\mathbb{Z}}

\newcommand{\caB}{\mathcal{B}}

\newcommand{\caD}{\mathcal{D}}

\newcommand{\caO}{\mathcal{O}}
\newcommand{\caP }{\mathcal{P}}

\newcommand{\caS}{\mathcal{S}}

\newcommand{\caW}{\mathcal{W}}

\def\ad {\mathop{\hbox {ad}}\nolimits}

\def \det {\mathrm{det}}

\def\Ind {\mathop{\hbox{Ind}}\nolimits}
\renewcommand{\ker}{\mathrm{ker}}

\usepackage[mathscr]{eucal}

\begin{document}


\title[Stable Combinations]
{Stable Combinations of Special Unipotent Representations}
\author{Dan M.~Barbasch}
\author{Peter E.~Trapa}

\address{Department of Mathematics, Cornell University, Ithaca, NY 14853, USA} 

\email{dmb14@cornell.edu}
\thanks{The first author was partially supported by NSF grants DMS-0967386
and DMS-0901104}

\address{Department of Mathematics, University of Utah, Salt Lake
  City, UT 84112, USA} 

\email{ptrapa@math.utah.edu} 

\thanks{The second author was partially supported by 
  NSF grants DMS-0554118 and DMS-0968060.}

\dedicatory{To Gregg Zuckerman, with respect and admiration}

\begin{abstract}
We define and study a class of virtual characters which
are stable in the sense of Langlands and Shelstad.  These
combinations are associated to nonspecial nilpotent orbits 
in certain even ``special pieces'' of
the Langlands dual,
and are defined in terms of characteristic cycles of perverse sheaves
on dual partial flag varieties.  Our results
generalize  earlier work of Adams, Barbasch, and Vogan.
\end{abstract}

\maketitle
\section{Introduction}
\label{s:intro}

In \cite{a1}--\cite{a2}, Arthur outlined a set of conjectures describing
the automorphic spectrum of semisimple Lie group over a local field.  
He suggested that the set of automorphic representations
is arranged into (possibly overlapping) packets satisfying
a number of properties.  In particular, each packet was predicted
to give rise to a canonical linear combination of its elements 
whose
character was stable in the sense of Langlands and Shelstad
\cite{lang}, \cite{lang-shel}.

In the real case, Arthur's predictions are made precise, refined, and
in many cases established in \cite{bv} and, most completely, 
in \cite{abv}.  Many of Arthur's conjectures can be reduced to the
case of a certain (precisely defined) set of special unipotent representations.
This set is  a union of Arthur packets, and since each
Arthur packet gives rise to a stable virtual representation, one thus
obtains a collection of stable linear combinations of special
unipotent representations.  One is naturally led to ask if their span
exhausts the space of stable virtual special unipotent
representations.  Simple examples show this is too naive.  For
example, in the complex case (where stability is empty) Arthur packets
are typically not singletons.  So the question becomes: can one give
a canonical basis of stable virtual special unipotent representations
which accounts for these ``extra'' stable sums?

Under certain natural hypotheses we give a positive answer in terms
of the geometry of ``special pieces'' of nilpotent cone of the Langlands dual Lie algebra.  
Recall (from \cite{spalt}) that if $\caO'$ is a nilpotent adjoint orbit
for a complex reductive Lie algebra, there is a unique special orbit $\caO$
of smallest dimension which contains $\caO'$ in its closure.  The collection of
all $\caO'$ for which $\caO$ is this unique orbit is called the special piece of
nilpotent cone parametrized by $\caO$.  We denote it $\SP(\caO)$.  The special
pieces form a partition of the set of nilpotent adjoint orbits indexed by
the special orbits.

In order to formulate our main results, like Theorem \ref{t:main} below, a number of
technicalities must be treated with care.  A significant complication,
as in \cite{abv}, is that one cannot work with a single real form 
individually, but instead must work with an inner class of them
simultaneously.  (We begin recalling the relevant details in Section
\ref{s:unip}.)  
In spite of these technicalities, some consequences of our results are easy to state and
have nothing to do with real groups.  For example, suppose $\frg$ is a complex
semisimple Lie algebra with adjoint group $G$.  Fix a Cartan subalgebra $\frh$ and a system of positive
roots $\Delta^+ = \Delta^+(\frh,\frg)$, and write $\frn = \bigoplus_{\alpha \in \Delta^+} \frg_\alpha$.
Write $W$ for the Weyl group of $\frh$ in $\frg$.  For $w \in W$, write
$\frn^w = \bigoplus_{\alpha \in \Delta^+} \frg_{w\alpha}$.  Then for each $w \in W$ there is always a dense nilpotent adjoint
orbit $\caO(w)$ contained in 
\[
G\cdot(\frn \cap \frn^w).
\]
A closely related variant of the map
$w \mapsto \caO(w)$ was studied by Steinberg in~\cite{steinberg}.
It is natural to ask if the map admits a canonical section.
That is, given a nilpotent adjoint orbit $\caO$, can one canonically
define a Weyl group element $w \in W$ such that $\caO$ is dense in
$G\cdot(\frn \cap \frn^w)$?

For example, suppose $\caO$ is even in the sense all of the labels on the
associated weighted Dynkin diagram are even.  Let $\frl = \frl(\caO)$ denote the subalgebra of $\frg$ corresponding
to the roots labeled zero.  If $w_\frl$ denotes the long element of the Weyl group of $\frl$, $W(\frl) \subset
W$, then indeed $\caO$ is dense in $G\cdot (\frn \cap \frn^{w_\frl})$.  
As a consequence of Corollary
\ref{c:Qcan} below (applied to the diagonal symmetric subgroup
in $G \times G$), we have the following generalization.

\begin{theorem}
\label{t:intro}
Suppose $\caO$ is an even nilpotent 
adjoint orbit for $\frg$.  Let $W^\frl$ denote
the set of maximal length coset representatives of $W(\frl) \bs
W/W(\frl)$ where $\frl = \frl(\caO)$ corresponds to the nodes labeled
zero in the weighted Dynkin diagram for $\caO$.  Let $d(\caO)$ denote
the Spaltenstein dual of $\caO$ (e.g.~\cite[Appendix B]{bv} or
\cite[Section 6.3]{CM}).  Assume that both $\caO$ and $d(\caO)$ are
even, and fix an adjoint orbit $\caO'$ in $\SP(\caO)$.

\begin{enumerate}

\item[(a)]
There exists
a unique element $w(\caO')\in W^\frl$ such that $\caO'$ is dense in
$G\cdot(\frn \cap \frn^{w(\caO')})$.  (For example, if $\caO' = \caO$, then
$w(\caO') = w_\frl$, the longest element in the identity coset.)

\item[(b)]  Let $\pi'$ denote
the Springer representation associated to the trivial local system on $\caO'$, and
let $\sgn$ denote the sign representation of $W(\frl)$.  Then
\[
\dim \Hom_{W(\frl)} \left ( \sgn, \pi' \right ) = 1.
\]

\end{enumerate}
\end{theorem}

\noindent 
Under the conditions of the theorem, the map $\caO' \mapsto w(\caO')$
in part (a) is thus a natural section of the map $w \mapsto \caO(w)$.
(The equivalence of statements (a) and (b) goes back to
Borho-MacPherson.  A more general statement is given in Proposition
\ref{p:geomcount} below.)  It would be interesting to investigate how
to relax the evenness hypotheses in the theorem.

The paper is organized as follows. 
After recalling the machinery of \cite{abv} in Section \ref{s:unip},
we state our main result in Theorem \ref{t:main}.  We prove it in the final two sections. 
Examples \ref{ex:main},  \ref{ex:mainf4}, and \ref{ex:maine8} illustrate many of the main ideas.

\bigskip

\noindent{\bf Acknowledgements.}
We thank Jeffrey Adams for drawing our attention to the problem considered in this paper.  
In particular, using the software package {\tt atlas}
he computed a basis for the space of stable virtual
special unipotent representations in many exceptional examples;
see {\tt www.liegroups.org/tables/unipotent}.
These examples led us to the formulation of 
Theorem \ref{t:main}.

Finally, it is a pleasure to dedicate this paper to Gregg Zuckerman.
His revolutionary ideas, particularly the construction of
cohomological induction and his approach to the
character theory of real reductive groups
(and its relation with tensoring with finite-dimensional
representations), are the foundations on which the results in this paper
are built.

\section{Statement and Examples of the Main Results}
\label{s:unip}

Let $G$ be a connected reductive complex algebraic.
We begin by fixing
a weak extend group $G^\Gamma$ for $G$ as in \cite[Definition 2.13]{abv}.
This means that there is an exact sequence of real Lie groups
\[
1 \lra G  \lra G^\Gamma  \lra \Gamma := \mathrm{Gal}(\bbC/\bbR) \lra 1
\]
and each $\delta \in G^\Gamma - G$ acts by conjugation as an antiholomorphic
automorphism of $G$.   If $\delta \in G^\Gamma - G$ is such that
$\delta^2 \in Z(G)$ --- that is if $\delta$ is a strong real form 
for $G^\Gamma$ in the language of \cite{abv} --- then conjugation by $\delta$ defines an antiholomorphic
involution of $G$.  In this case, we write $G(\bbR, \delta)$ for the corresponding
fixed points, a real form of $G$.  It follows from \cite[Proposition 2.14]{abv}
that the set of real forms which arise in this way constitute exactly one
inner class of real forms, and moreover every such inner class arises in this way.
In particular, by fixing $G^\Gamma$ we have fixed
an inner class of real forms of $G$.

Recall (again from \cite[Definition 2.13]{abv}) that a representation of a strong real
form for $G^\Gamma$ is a pair ($\pi, \delta)$ where $\delta$ is a strong real form
of $G^\Gamma$ and $\pi$ is an admissible representation of $G(\bbR, \delta)$.  
Two representations
$(\pi,\delta)$ and $(\pi',\delta')$ are said to be equivalent if
there is an element $g \in G$ such that $\delta' = g\delta g^{-1}$ and
$\pi'$ is infinitesimal equivalent to  $\pi\circ Ad(g^{-1})$.  Write $\Pi(G/\bbR)$
for the set of infinitesimal equivalence classes of irreducible representations of strong
real forms for $G^\Gamma$.

\medskip
Fix a maximal ideal $I$ in the center of the enveloping algebra
$\U(\frg)$ of the Lie algebra $\frg$ of $G$.  Choose a Cartan
subalgebra $\frh \in \frg$ and write $W$ for the Weyl group of $\frh$
in $\frg$.  According to the Harish-Chandra isomorphism we may attach
an element $\lam \in \frh^*/W$ to $I$.  Let $\Pi^\lam(G/\bbR)$ denote the subset
of $\Pi(G/\bbR)$ consisting of those representations whose
associated Harish-Chandra modules are annihilated by $I$.  Write
$\bbZ\Pi^\lam(G/\bbR)$ for the (finite rank) free $\bbZ$ module with
basis indexed by $\Pi^\lam(G/\bbR)$.

\medskip
We next introduce various objects on the dual side.  Let $G^\vee$ denote the
Langlands dual group corresponding to $G$, and write $\frg^\vee$ for its Lie algebra.
Let $\wt G^\vee_\alg$ denote the algebraic universal cover of $G^\vee$ (e.g.~\cite[Definition 1.18]{abv}).  
For later use, recall that the construction of the dual group specifies a Cartan subalgebra $\frh^\vee$ of
$\frg^\vee$ which is canonically isomorphic to $\frh^*$.

Definition 1.8 and Lemma 1.9 of \cite{abv} introduce a smooth complex
algebraic variety $X = X(G^\Gamma)$ attached to the extended group
fixed above, and provide an action of $\wt G_\alg^\vee$
on $X$ which factors to an action of $G^\vee$.
(To be more precise, \cite[Definition 1.8]{abv}
explains how to define $X$ from an $L$-group, and the discussion
around \cite[Proposition 4.14]{abv} explains how to build an $L$-group
from a fixed inner class of real forms, in particular the class
specified by our fixed weak extended group $G^\Gamma$.)  The variety
$X$ is a disjoint union of smooth (possibly empty) finite-dimensional
varieties $X^\lam$ indexed by $\lam \in \frh^*/W$.  The action of
$G^\vee$ (and $\wt G^\vee_\alg$) on $X$ preserves each $X^\lam$.  The
orbits for both actions on $X^\lam$ are the same and are finite in
number.  We do not recall the general structure of $X^\lam$ here, but
instead describe certain special cases in detail below.

Let $\caP(X^\lam,\wt G_\alg^\vee)$ denote the category of $\wt G_\alg^\vee$-equivariant perverse sheaves on 
$X^\lam$, and write $\bbZ\caP(X^\lam,\wt G_\alg^\vee)$ for its integral Grothendieck group.
Let $T^*_{G^\vee}(X^\lam)$ denote the conormal variety for the action of $G^\vee$ on 
$X^\lam$, namely the subvariety of $T^*(X^\lam)$ consisting of the unions of the 
various conormal bundles $T_Q^*(X^\lam)$ to $G^\vee$ orbits $Q$ on $X^\lam$.  
(Recall that the orbits of $G^\vee$ and
$\wt G^\vee_\alg$ are the same.)  The characteristic cycle functor
gives a map
\[
\CC \; : \; \bbZ\caP(X^\lam,\wt G_\alg^\vee) \lra 
H_\top\left (T^*_{G^\vee}(X^\lam), \bbZ\right ) \simeq \bigoplus_{Q \in G^\vee \bs X^\lam}
\bbZ\left [ \ol{T_Q^*(X^\lam)} \right ].
\]
The right-hand side is the top-dimensional integral Borel-Moore homology group 
of $T^*_{G^\vee}(X^\lam)$ which, as indicated, is isomorphic to the direct sum of the $\bbZ$ span
of the fundamental
classes of closures of the individual conormal bundles.

The ABV interpretation of the Local Langlands Conjecture, summarized in \cite[Corollary 1.26]{abv},
provides a $\bbZ$-module isomorphism
\[
\Phi \; : \; \bbZ\Pi^\lam(G/\bbR) \simeq \left (\bbZ\caP(X^\lam,\wt G_\alg^\vee)\right )^\star
\]
for each $\lam \in \frh^*/W$; here and elsewhere 
$(\; \cdot\; )^\star$ applied to a $\bbZ$-module
denotes $\Hom_\bbZ(\; \cdot\; ,\bbZ)$.
The isomorphism $\Phi$ depends on more data than just the weak extended group $G^\Gamma$ fixed above.
It requires fixing a (strong) extended group $(G^\Gamma, \caW)$ as in
\cite[Definition 1.12]{abv} and a strong real form \cite[Definition 1.13]{abv}.
We define
\begin{equation}
\label{e:stable}
\bbZ_\st\Pi^\lam(G/\bbR) 
:= \Phi^{-1} \left ( \bbZ\caP(X^\lam,\wt G_\alg^\vee) \bigr /
  \ker(\CC) \right)^\star. 
\end{equation}
This is a space of integral linear combinations of irreducible representations 
of $G^\Gamma$, that is virtual representations.
(This space depends only on the weak extended group $G^\Gamma$.)

For the purpose of this paper, we may take \eqref{e:stable} 
as the definition of the
subspace of stable virtual characters in $\bbZ\Pi^\lam(G/\bbR)$.
The equivalence with Langlands' original formulation of stability 
is given in \cite[Chapter 18]{abv}.

The main aim of this paper is to define a canonical basis of
$\bbZ_\st\Pi^\lam(G/\bbR)$  (in certain special cases) indexed
by rational forms of special pieces of the nilpotent cone of $\frg^\vee$.
We now specify the
special cases of interest.  Begin by fixing a nilpotent adjoint orbit
$\caO^\vee$ for $\frg^\vee$.  Choose a Jacobson-Morozov triple $\{e^\vee, f^\vee, h^\vee\}$
with $h^\vee \in \frh^\vee$ ($\frh^\vee$ as defined above).  Set
\begin{equation}
\label{e:unipic}
\lam(\caO^\vee) = \frac12 h^\vee \in \frh^\vee \simeq \frh^*.
\end{equation}
Define
\begin{equation}
\label{e:unipl}
\frl^\vee(\caO^\vee) = \text{ the centralizer in $\frg^\vee$ of $\lam(\caO^\vee)$};
\end{equation}
equivalently $\frl^\vee(\caO^\vee)$ is the sum of the zero eigenspaces of $\ad(h^\vee)$.
Set
\begin{equation}
\label{e:unipp}
\frp^\vee(\caO^\vee) = \text{ the sum of the non-negative eigenspaces of $\ad(h^\vee)$.}
\end{equation}
Let $I(\caO^\vee)$ denote the maximal ideal in the center of $\U(\frg)$ 
corresponding to $\lam(\caO^\vee)$ under the Harish-Chandra isomorphism.
According to a result of Dixmier \cite{dix}, 
there is a unique maximal primitive ideal
$J(\caO^\vee)$ in $\U(\frg)$ containing $I(\caO^\vee)$.  We say a representation
$(\delta, \pi)$ of $G^\Gamma$ is {\em special unipotent attached to $\caO^\vee$}
if the Harish-Chandra module of $\pi$ is annihilated by $J(\caO^\vee)$.
We write
\[
\Pi(\caO^\vee) \subset \Pi^{\lam(\caO^\vee)}(G/\bbR)
\]
for the subset of irreducible special unipotent representations of $G^\Gamma$
attached to $\caO^\vee$, write $\bbZ\Pi(\caO^\vee)$ for their span, and define
\[
\bbZ_\st\Pi(\caO^\vee) := \bbZ\Pi(\caO^\vee) \cap \bbZ_\st\Pi^\lam(G/\bbR).
\]
It is this space for which we will find a canonical basis under certain natural 
hypotheses.

To state our main results, we need more detailed information about the
structure of the $G^\vee$ action on $X^\lam$ assuming $\lam$ is integral.  
Let 
\begin{equation}
\label{e:unipY}
Y^\vee = \text{the variety of parabolic subalgebras of $\frg^\vee$ conjugate
to $\frp^\vee(\caO^\vee)$}
\end{equation}
with notation as in \eqref{e:unipp}.
Proposition 6.16 of \cite{abv}
provides a collection of symmetric subgroups $K^\vee_1, \dots, K^\vee_k$ of $G^\vee$.
Each $K_i^\vee$ acts on $Y^\vee$ with finitely many orbits.  Furthermore,
\cite[Proposition 7.14]{abv} implies the existence of an isomorphism
\begin{equation}
\label{e:Ki}
\caP(X^\lam, G^\vee) \; \simeq \; \caP(Y^\vee, K_1^\vee) \oplus \cdots \oplus \caP(Y^\vee, K_k^\vee),
\end{equation}
where $\caP(Y^\vee, K_i^\vee)$ once again denotes the category of
$K_i^\vee$ equivariant perverse sheaves on $Y^\vee$.  Moreover, if we
let $\CC_i$ denote the characteristic cycle functor for $\caP(Y^\vee,
K_i^\vee)$, then the isomorphism in \eqref{e:Ki} descends to an
isomorphism
\begin{equation}
\label{e:ccisom1}
\caP(X^\lam, G^\vee)\bigr / \ker(\CC)
 \; \simeq \; \caP(Y^\vee, K_1^\vee)  \bigr / \ker(\CC_1)\; \oplus \cdots \oplus \; \caP(Y^\vee, K_k^\vee) \bigr / \ker(\CC_k).
\end{equation}
General properties of the characteristic cycle construction imply that it is insensitive to central
extensions of the group acting in the sense that
\[
\caP(X^\lam, \wt G_\alg^\vee)\bigr / \ker(\CC) \simeq 
\caP(X^\lam, G^\vee)\bigr / \ker(\CC);
\]
see \cite[Proposition 2.6.2]{chang}, for example.
Thus \eqref{e:ccisom1} in fact gives
\begin{equation}
\label{e:ccisom2}
\caP(X^\lam, \wt G_\alg^\vee)\bigr / \ker(\CC)
 \simeq \caP(Y^\vee, K_1^\vee) \bigr / \ker(\CC_1) \oplus \cdots \oplus \caP(Y^\vee, K_k^\vee) \bigr / \ker(\CC_k).
\end{equation}
As a matter of notation, we let $\frk^\vee_i$ denote the Lie algebra of $K_i^\vee$ and write
\begin{equation}
\label{e:cartani}
\frg^\vee = \frk^\vee_i \oplus \frs^\vee_i
\end{equation}
for the corresponding Cartan decomposition.
According to \cite{kr}, if $\caO^\vee$ is any nilpotent adjoint orbit in $\frg^\vee$, then 
each $K_i^\vee$ acts with finitely many orbits on $\caO^\vee \cap \frs^\vee_i$,
\[
\#\left ( K_i^\vee \backslash (\caO^\vee \cap \frs^\vee_i) \right ) < \infty.
\]

Recall that an orbit $\caO^\vee$ for $\frg^\vee$ is said to be even, 
if the 
eigenvalues of $\ad(h^\vee)$ acting on $\frg^\vee$ are all even
integers; equivalently if $\lam(\caO^\vee)$ is integral.  Assume this is the case and fix an
orbit
\[
\caO^\vee_K  \; \in \; \bigcup_{i = 1}^k K_i^\vee \backslash (\caO^\vee \cap \frs^\vee_i).
\]
Chapter 27 of \cite{abv} defines an Arthur packet parametrized by $\caO^\vee_K$,
\[
\mathrm{A}(\caO^\vee_K) \subset \Pi(\caO^\vee).
\]
The union of the various Arthur packets (over all possible orbits $\caO^\vee_K$)
exhausts $\Pi(\caO^\vee)$ (but the union is not in general disjoint).  Moreover,
for each $\caO^\vee_K$, 
the discussion around \cite[(1.34c)]{abv} defines a stable integral linear combination of elements of 
$\Pi(\caO^\vee_K)$,
\begin{equation}
\label{e:abvstab}
\pi(\caO^\vee_K) \in \bbZ_\st\Pi(\caO^\vee).
\end{equation}
These virtual representations are all linearly independent, so in particular one has
\begin{equation}
\label{e:premaineq}
\dim_\bbZ \bbZ_\st\Pi(\caO^\vee) \geq \sum_i \#\left ( K_i^\vee \backslash (\caO^\vee \cap \frs^\vee_i) \right ).
\end{equation}
Our main result
finds other interesting stable representations attached to $K^\vee$ orbits
on the special piece parametrized by $\caO^\vee$ and (in favorable instances) proves they
are a basis of $\bbZ_\st(\caO^\vee)$.

\begin{theorem}
\label{t:main}
Let $G$ be a connected reductive algebraic group with dual group $G^\vee$.  Fix a weak extended group
$G^\Gamma$ for $G$ (in particular, an inner class of real forms for $G$).
Fix an even nilpotent adjoint orbit $\caO^\vee$ for $\frg^\vee$.
Assume further that the Spaltenstein dual $\caO:= d(\caO^\vee)$, a nilpotent adjoint orbit for $\frg$
(cf.~\cite[Appendix B]{bv}), is also even.

Recall the Cartan decompositions of \eqref{e:cartani}
and the corresponding symmetric subgroups $K_i^\vee$ introduced above.  Write $\SP(\caO^\vee)$ for
the special piece of the nilpotent cone of $\frg^\vee $ containing  $\caO^\vee$.
Then 
\begin{equation}
\label{e:maineq}
\dim_\bbZ \bbZ_\st\Pi(\caO^\vee) = \sum_i \#\left ( K_i^\vee \backslash (\SP(\caO^\vee) \cap \frs^\vee_i) \right );
\end{equation}
cf.~\eqref{e:premaineq}.
In fact, for each element $\caO^\vee_K$ on the right-hand side, 
equation \eqref{e:extradef} below defines an element
$\pi(\caO^\vee_K) \in  \bbZ_\st\Pi(\caO^\vee)$ so that
\[
\left \{ \pi(\caO^\vee_K) \; \bigr | \; \caO^\vee_K \in \bigcup_i K^\vee_i \bs (\SP(\caO^\vee) \cap \frs^\vee_i)\right \}
\]
is a basis of $\bbZ_\st\Pi(\caO^\vee)$.  (When $G^\vee \cdot
\caO^\vee_K = \caO^\vee$, $\pi(\caO_K^\vee)$ coincides with the stable virtual representation  in
\eqref{e:abvstab}.)
\end{theorem}

\begin{example}
\label{ex:main}
Suppose $G = \Sp(4,\C)$ and $G^\Gamma$ gives rise to the inner class of $G$
containing the split form.   
There are four equivalence classes of strong real forms 
for $G^\Gamma$, $\{\delta_s, \delta_{2,0}, \delta_{1,1}, \delta_{0,2}\}$.  The labeling is arranged so that
$G(\bbR, \delta_s) = \Sp(4,\bbR)$ and $G(\bbR, \delta_{p,q}) = \Sp(p,q)$.  

Let $\caO^\vee$ denote the (even) nilpotent orbit for $G^\vee =
\SO(5,\C)$ whose Jordan type is given by the partition $311$.  
Then $d(\caO^\vee)$ is the orbit for $G$ with Jordan type corresponding to the partition 
$22$, which is also even, so Theorem \ref{t:main}
applies.

In \cite[Example 27.14]{abv}, the elements of $\Pi(\caO^\vee)$ are enumerated. 
Among them are eight representations of $\Sp(4,\R)$ and one of $\Sp(1,1)$.  
The representations of $\Sp(4,\R)$ are the three irreducible constituents of 
$\Ind_{GL(2,\bbR)}^{\Sp(4,\R)}(\det)$; the three irreducible constituents of
 $\Ind_{GL(2,\bbR)}^{\Sp(4,\R)}(|\det|)$; and the two irreducible constituents of
 $\Ind_{GL(1,\bbR) \times \Sp(2, \bbR)}^{\Sp(4,\bbR)}(\sgn(\det) \otimes 1)$.
 These eight representations are distinguished by their lowest $\U(2)$ types
 which in the respective cases are $(2,0), (0,2)$, and $(0,0)$; $(1,1), (-1,-1)$, and
 $(1,-1)$; and $(1,0)$ and $(0,-1)$.  Write $\pi_s(m,n)$ for the corresponding 
 special unipotent representation of $\Sp(4,\R)$ with lowest
 $\U(2)$ type $(m,n)$.  Meanwhile the unique special unipotent representation
 of $\Sp(1,1)$ attached to $\caO^\vee$ is the irreducible spherical representation
 with infinitesimal character $\lam(\caO^\vee)$ which we denote by $\pi_{(1,1)}(0)$.

The symmetric
subgroups $K_i^\vee$  above in this case are
\[
K^\vee_i = \mathrm{S}(\mathrm{O}(5-i) \times \mathrm{O}(i)) \text{ for $i=0,1,2$}.
\]
In terms of the signed tableau parametrization (for example, \cite[Chapter 9]{CM}), 
$\bigcup_i K^\vee_i\bs(\caO^\vee \cap \frs^\vee)$ consists of three elements
\[
\begin{matrix}
  +&-&+\\ +\\ +
\end{matrix}
\qquad , \quad
\begin{matrix}
  -&+&-\\ +\\ +
\end{matrix}
\qquad , \quad
\begin{matrix}
  +& -&+\\ +\\ -
\end{matrix}
\qquad ;
\]
the first arises for $i=1$, the second and third for $i=2$.  This means there are three Arthur packets
in $\Pi(\caO^\vee)$.  They are listed in \cite[(27.17)]{abv}.  They give rise, respectively, to 
the following three stable virtual representations
in $\bbZ_\st(\caO^\vee)$,
\begin{align*}
&\pi_s(1,0)+\pi_s(0,-), \\
&\pi_s(0,0)+\pi_s(1,-1), \\
&\pi_s(1,1)+\pi_s(-1,-1) +\pi_s(2,2)+\pi_s(-2,-2)+\pi_{(1,1)}(0).
\end{align*}
Meanwhile there is another orbit $\caO^{\vee'}$ 
(besides $\caO^\vee$) in $\SP(\caO^\vee)$, namely the orbit with Jordan type corresponding to
the partition $221$.  
This time $\bigcup_i K^\vee_i\bs(\caO^{\vee'} \cap \frs^\vee)$ consists of  one element
\begin{equation}
\label{e:extra}
\begin{matrix}
  +& - \\-& +\\ +
\end{matrix}\end{equation}
arising for $K^\vee_2$.
Theorem \ref{t:main} thus implies $\dim_\bbZ\bbZ_\st(\caO^\vee) = 3+1$, and gives
an additional
stable virtual representation parametrized by the orbit in \eqref{e:extra}.
This extra stable sum is
\begin{equation}
  \label{eq:exsp4nsp}
\pi_s(1,1)+\pi_s(-1,-1)-\pi_s(2,2)-\pi_s(-2,-2).
\end{equation}
\end{example}

\begin{example}
\label{ex:mainf4}
Let $G$ be of type F4, and let $\caO^\vee$ be the orbit labeled
$F4(a3)$ in the Bala-Carter classification (e.g.~\cite[Section 8.4]{CM}).  
If we orient the Dynkin diagram of F4 so that the first two roots are long,
the weighted Dynkin diagram for $\caO^\vee$ is $0200$.  In particular, the orbit
is even.  In fact $\caO^\vee$ is equal to its own Spaltenstein
dual, and thus Theorem \ref{t:main} applies.

The special piece $\SP(\caO^\vee)$ consists of four other orbits
besides $F4(a3)$.  In
the Bala-Carter classification, they are labeled $C3(a1), \wt{A2} + A1,
B2$, and $A2 + \wt{A1}$.  The respective weighted Dynkin diagrams are
$1010,0101,2001,$ and $0010$.

There is a unique inner
class of real forms for $G$; it contains the split, rank one, and
compact forms.  
(In fact it is easy to see (from the singularity of the infinitesimal character
$\lam(\caO^\vee)$ that $\Pi(\caO^\vee)$ can consist
of representations only of the split form.)
The only symmetric subgroup $K^\vee$ appearing above
in this case is  the quotient of $\Sp(6,\C) \times SL(2,\C)$ by
the diagonal copy of a central $\bbZ/2$.  From the tables in 
\cite[Section 9.5]{CM}, we deduce that
\[
\# K^\vee \bs (F4(a3) \cap \frs^\vee) = 3
\]
and so there are three Arthur packets in $\Pi(\caO^\vee)$.  Meanwhile
we have
\begin{align*}
&\# K^\vee \bs (C3(a1) \cap \frs^\vee) = 2\\
&\# K^\vee \bs ((\wt{A2} + A1) \cap \frs^\vee) = 1\\
&\# K^\vee \bs (B2 \cap \frs^\vee) = 2\\
&\# K^\vee \bs ((A2 + \wt{A1}) \cap \frs^\vee) = 1.\\
\end{align*}
Thus Theorem \ref{t:main} says
\[
\dim_\bbZ \bbZ_\st\Pi(\caO^\vee) = 3 + 2 + 1 + 2 + 1.
\]
The definition in \eqref{e:extradef} gives a canonical basis
for the space.  To write the basis down explicitly requires
computing characteristic cycles of irreducible objects
in $\caP(K^\vee, Y^\vee)$.  We have not performed the 
calculations required to do this.
\end{example}

\begin{example}
\label{ex:maine8}
Let $G$ be of type E8, and let $\caO^\vee$ denote the
orbit labeled $E8(a7)$ in the Bala-Carter classification.
It is even and self-dual, and thus Theorem \ref{t:main}
applies.  The special piece it parametrizes consists of the
additional orbits $E7(a5), E6(a3) + A1, D6(a2), D5(a1)+A2,
A5+A1,$ and $A4+A3$.

There is a unique inner
class of real forms for $G$, and (arguing as in the
previous example), $\Pi(\caO^\vee)$ can consist
of representations only of the split form.
The only symmetric subgroup $K^\vee$ appearing above
in this case is a quotient of $\mathrm{Spin}(16,\C)$ by
a central $\bbZ/2$ (but $K^\vee$ is not isomorphic to $\mathrm{SO}(16,\C)$).  Again using
the tables in 
\cite[Section 9.5]{CM}, we deduce that
\[
\# K^\vee \bs (E8(a7) \cap \frs^\vee) = 3
\]
and so there are three Arthur packets in $\Pi(\caO^\vee)$.  Meanwhile
we have
\begin{align*}
&\# K^\vee \bs (E7(a5) \cap\frs^\vee) = 2 \\
&\# K^\vee \bs ((E6(a3)+A1)\cap\frs^\vee) = 2 \\
&\# K^\vee \bs (D6(a2)\cap\frs^\vee) = 2 \\
&\# K^\vee \bs ((A5 + A1)\cap\frs^\vee) = 1 \\
&\# K^\vee \bs  ((D5(a1) + A2)\cap\frs^\vee) = 1 \\
&\# K^\vee \bs ((A4+A3)\cap\frs^\vee) = 1.
\end{align*}
Thus Theorem \ref{t:main} implies
\[
\dim_\bbZ \bbZ_\st\Pi(\caO^\vee) = 3 + 2 + 2 + 2 + 1 + 1 +1.
\]
\end{example}

\section{Proof of Equality in \eqref{e:maineq}}
\label{s:proof}

Our main technique allows us to compute the numbers in \eqref{e:maineq} in terms
of certain Weyl group calculations.  The full Weyl group does not act at singular infinitesimal character,
and so we must instead translate to regular infinitesimal character and work there.

Retain the setting of Theorem \ref{t:main}.
Temporarily choose a system of simple roots for $\frh$ in $\frg$ and a representative
$\lam_\circ$ of $\lam(\caO^\vee)$ which is dominant.  Let $\mu$ be the highest weight
of a finite-dimensional representation of $G$ 
such that $\nu_\circ:= \lam_\circ + \mu \in \frh^*$ is dominant and regular.  Let
$\nu$ denote the image of $\nu_\circ$ in $\frh^*/W$.  As above, we can consider
the set $\Pi^\nu(G/\bbR)$ and its $\bbZ$ span $\bbZ\Pi^\nu(G/ \bbR)$.  This space identifies with an
appropriate Grothendieck group of representations at regular integral infinitesimal character which admits a coherent continuation action of $W$.  

Recall the symmetric subgroups $K^\vee_1, \dots, K^\vee_k$ of the previous section.  Let $X^\vee$ denote the full flag variety for
$\frg^\vee$.  There is an action of $W$ on each Grothendieck group $\bbZ\caP(X^\vee,K^\vee_i)$. 
(One way to see this is to use the Riemann-Hilbert correspondence to identify $\caP(X^\vee, K^\vee_i)$
with a category of $K^\vee_i$ equivariant holonomic $\caD$ modules on $X^\vee$.  In turn, by localization,
this category is a equivalent to a category of $\frg^\vee$ modules which admits a coherent continuation action
of $W$.)
 
Meanwhile Corollary 1.26 and Proposition 7.14 of \cite{abv} give an isomorphism (depending on
the extended group $(G^\Gamma, \caW)$)
\[
\Psi \; : \; \bbZ\Pi^\nu(G/\bbR) \lra \bigoplus_{i} \bbZ\caP(X^\vee,K^\vee_i)^\star
\]
which intertwines the $W$ action on both sides.
Once again we have characteristic cycle functors
\[
\CC_i \; : \; \bbZ\caP(X^\vee,K^\vee_i) \lra \H_\top\left( T_{K^\vee_i}^*(X^\vee), \bbZ\right) 
\simeq \bigoplus_{Q \in K^\vee_i \bs X^\vee} \left [\ol{T^*_{K^\vee_i}(X^\vee)}\right ].
\]
We have remarked that the domain of $\CC_i$ carries an action of $W$.  The range does as
well, and according to results of Tanisaki \cite{tan}, each $\CC_i$ is $W$-equivariant.  Thus
\[
\bbZ\caP(X^\vee,K^\vee_i)\bigr / \ker(\CC_i) \simeq \H_\top\left( T_{K^\vee_i}^*(X^\vee), \bbZ\right) 
\]
as representations of $W$.  Once again we define
\[
\bbZ_\st\Pi^\nu(G/ \bbR) := \Psi^{-1}\left ( \bigoplus_i \bbZ\caP(X^\vee,K^\vee_i)\bigr / \ker(\CC_i) \right )^\star.
\]
and we have an isomorphism
\begin{equation}
\label{e:regstableisom}
\bbZ_\st\Pi^\nu(G/ \bbR) \simeq \bigoplus_i \H_\top\left( T_{K^\vee_i}^*(X^\vee), \bbZ\right)^\star
\end{equation}
of representations of $W$.

For our counting argument, we need to specify a particular left cell representation.
Let $\frl^\vee(\caO^\vee)$ denote the centralizer in $\frg^\vee$ of
$\lam(\caO^\vee) \in \frh^\vee$, and let
$w(\caO^\vee)$ denote the long element of the Weyl group of
$\frl^\vee(\caO^\vee)$ viewed as an element of $W(\frh^\vee, \frg^\vee) = W$.  
Let
$V(\caO^\vee)$ denote the representation of $W$ afforded by the
integral linear combinations of elements of the Kazhdan-Lusztig
left cell containing
$w(\caO^\vee)$. 

\begin{prop}[{\cite[Section 5]{bv}}]
\label{p:sunipcount}
Retain the setting above.  In particular, assume
$\caO^\vee$ is even.
We have
\[
\dim_\bbZ\bbZ\Pi(\caO^\vee)
= \dim \Hom_W(V(\caO^\vee) \otimes \sgn, \bbZ\Pi^\nu(G/ \bbR)).
\]
and
\[
\dim_\bbZ\bbZ_\st\Pi(\caO^\vee)
= \dim \Hom_W(V(\caO^\vee) \otimes \sgn, \bbZ\Pi_\st^\nu(G/ \bbR)).
\]
\end{prop}

\medskip

The following result brings the role of special pieces into play.  To state it,
we need to introduce some notation for the Springer correspondence.
Fix any nilpotent adjoint orbit $\caO$ for $\frg$ and a representative $x$ of
$\caO$.  Let $A_G(\xi)$ denote the component group of the centralizer
of $x$ in $G$.  We let $\Sp(x)$ denote the Springer representation
of $W \times A_G(x)$ on the top homology of the Springer fiber over
$x$ (normalized so that $\Sp(x)$ is the sign representation of $W$
if $x$ is zero).  As usual, we set
\[
\Sp(x)^{A_G(x)} = \Hom_{A_G(x)}\left(\triv, \Sp(\xi)\right).
\]
This is a 
a representation of $W$.  

\begin{prop}
\label{p:cellcomp1}
Suppose $\caO^\vee$ is a an even nilpotent adjoint orbit for $\frg^\vee$.
Let $\caO$ denote special nilpotent orbit for $\frg$ obtained as the Spaltenstein dual
of $\caO^\vee$.   Enumerate representative for the adjoint orbits in the special piece parametrized 
by $\caO$ as $x_1, \cdots, x_l$.  Then
\[
V(\caO^\vee) \otimes \sgn \simeq \bigoplus_{i}\Sp(x_i)^{A_G(x_i)}.
\]
\end{prop}
\begin{proof}
This follows by combining \cite[Proposition 5.28]{bv} and \cite[Theorem 0.4]{lu}.
\end{proof}

The proposition involves special pieces on the group side, while the statement
of Theorem \ref{t:main} involves special pieces on the dual side.   If we make the
additional hypothesis that $\caO$ is even, then we can match up the two sides.

\begin{prop}
\label{p:cellcomp2}
Suppose $\caO^\vee$ is a an even nilpotent adjoint orbit for $\frg^\vee$.
Let $\caO$ denote its Spaltenstein dual, and further assume that $\caO$ is even.
Enumerate representative in the special piece parametrized 
by $\caO^\vee$ as $x^\vee_1, \cdots, x^\vee_l$.  Then
\[
V(\caO^\vee) \simeq \bigoplus_{i} \Sp(x^\vee_i)^{A_{G^\vee}(x^\vee_i)}.
\]
That is, $V(\caO^\vee)$ is the sum over the orbits in $\SP(\caO^\vee)$
of the Weyl group representations attached to the trivial local system on them.
\end{prop}
\begin{proof}
This follows from the previous proposition and Lusztig's classification of cells \cite{lu:crg} (the relevant
details of which are recalled in \cite[Theorem 4.7d]{bv}).
\end{proof}

\bigskip

Proposition \ref{p:sunipcount} gives a way to compute the dimension of the left-hand side
of \eqref{e:maineq} in terms of Weyl group representations.  We need a way to do the
same for the right-hand side.

\begin{prop}
\label{p:rossmann}
Suppose $K^\vee$ is a symmetric subgroup of $G^\vee$ and write
$\frg^\vee = \frk^\vee \oplus \frs^\vee$ for the corresponding Cartan decomposition.
Let $x^\vee_1, x_2^\vee, \dots, $ denote representatives of the nilpotent $K^\vee$ orbits on 
$\frs^\vee$.   Let $A_{K^\vee}(x^\vee_i)$ denote the component group of the
centralizer in $K^\vee$ of $x^\vee_i$.  (Since this group maps to $A_{G^\vee}(x^\vee_i)$, it
makes sense to consider invariants $\Sp(x^\vee_i)^{A_{K^\vee}(x_i^\vee)}$ in $\Sp(x^\vee_i)$
of the image of $A_{K^\vee}(x^\vee_i)$ in $A_{G^\vee}(x^\vee_i)$.)
As $W$ representations, we have
\[
 \H_\top(T_{K^\vee}^*(X^\vee), \bbZ)
 \simeq \sum_{i} \Sp(x^\vee_i)^{A_{K^\vee}(x^\vee_i)}.
 \]
 In particular, each representation of $W$ attached to the trivial local system on a complex
 nilpotent orbit appears with multiplicity equal to the number of $K^\vee$ orbits on its
 intersection with $\frs^\vee$.
\end{prop}

\noindent{\bf Proof.}  
This follows from \cite[Theorem 3.3]{ro}.  (Rossmann works with the conormal variety of
orbits of real forms of $G^\vee$ on $X^\vee$.  To translate to the conormal variety of orbits for a symmetric
subgroup, one can use \cite{muv}, for example.)
\qed

\medskip

\noindent{\bf Proof of equality in \eqref{e:maineq}.}
{From} Proposition \ref{p:sunipcount} and \eqref{e:regstableisom}, we have
\begin{align*}
\dim_\bbZ \bbZ_\st\Pi(\caO^\vee) &= \dim \Hom_W\left (V(\caO^\vee) \otimes \sgn, \bbZ_\st\Pi^\nu(G/\bbR) \right ) \\
&= \dim \Hom_W\left (V(\caO^\vee) \otimes \sgn\; , \; \bigoplus_i \H_\top(T_{K_i^\vee}^*\caB, \bbZ)^* \right ) \\
&= \dim \Hom_W\left (V(\caO^\vee) \; , \; \bigoplus_i \H_\top(T_{K_i^\vee}^*\caB, \bbZ) \right ).
\end{align*}
The concluding sentences of  Propositions \ref{p:cellcomp2} and \ref{p:rossmann} show that the right-hand side 
equals
\[
\sum_i \#\left ( K_i^\vee \backslash (\SP(\caO^\vee) \cap \frs^\vee_i) \right )
\]
as claimed.
\qed

\section{Proof of Theorem \ref{t:main}}
In this section we prove the last assertion of Theorem \ref{t:main}.
According to \eqref{e:stable},
\begin{equation}
\label{e:first}
\bbZ_\st\Pi(\caO^\vee) \subset 
\bbZ_\st\Pi^{\lam(\caO^\vee)}(G/\bbR)
\simeq
\bigoplus_i \left [ \H_\top(T^*_{K^\vee_i}(Y^\vee),\bbZ) \right ]^\star.
\end{equation}
Our main task is to determine which linear functionals on
\begin{equation}
\label{e:fundclass}
\H_\top(T^*_{K^\vee_i}(Y^\vee),\bbZ) = \bigoplus_{Q \in K_i^\vee \bs Y^\vee} \left [ \ol{T_Q^*(Y^\vee)} \right ]
\end{equation}
correspond to elements of $\bbZ_\st\Pi(\caO^\vee)$ in \eqref{e:first}.   This is the content of part (2) of  the next proposition.
To formulate it, we recall the $G^\vee$ equivariant moment map $\mu$ from $T^*(Y^\vee)$ to $(\frg^\vee)^*$.
We use an invariant form to identify $\frg^\vee$ and $(\frg^\vee)^*$, and view the image of the moment
map in $\frg^\vee$ itself.

\begin{prop}
\label{p:functionals}
Retain the setting of Theorem \ref{t:main}.  For each orbit $Q$ of some $K_i^\vee$ on $Y^\vee$, define
\[
m_Q \in \left [ \H_\top(T^*_{K^\vee_i}(Y^\vee),\bbZ) \right ]^\star
\]
as the multiplicity of the fundamental class 
  corresponding to the closure of the conormal bundle to $Q$
(c.f.~\eqref{e:fundclass}). Recall the isomorphism
\begin{equation}
\label{e:stablefunctionals}
\bbZ_\st\Pi^{\lam(\caO^\vee)}(G/\bbR)
\simeq
\bigoplus_i \left [ \H_\top(T^*_{K^\vee_i}(Y^\vee),\bbZ) \right ]^\star
\end{equation}
and write $\pi(Q)$ 
for the element of $\bbZ_\st\Pi^{\lam(\caO^\vee)}(G/\bbR)$ corresponding to $m_Q$.
\begin{enumerate}
\item[(1)]
The set 
\begin{equation}
\label{e:stablefunctionals1}
\left \{ \pi(Q) \; | \; Q \in K^\vee_i \bs Y^\vee \text{ for some $i$} \right \}
\end{equation}
 is a basis
for 
\[
\bbZ_\st\Pi^{\lam(\caO^\vee)}(G/\bbR).
\]

\item[(2)]
The set
\begin{equation}
\label{e:stablefunctionals2}
 \left \{ \pi(Q) \; \bigr | \; \mu \left ( T^*_Q(Y^\vee) \right) \cap \SP(\caO^\vee)
\textrm{ is nonempty} \right \}
\end{equation}
is a basis
for the subspace
\[
\bbZ_\st\Pi(\caO^\vee) 
\subset \bbZ_\st\Pi^{\lam(\caO^\vee)}(G/\bbR).
\]
\end{enumerate}
\end{prop}

\begin{proof}
Since
\[
\bigcup_i \{ m_Q \; | \; Q \in K_i^\vee \bs Y^\vee \}
\]
is obviously a basis for the left-hand side of \eqref{e:stablefunctionals}
(in light of \eqref{e:fundclass}), statement (1)
of the proposition is clear.

For the second statement, we begin by proving
\[
 \left \{ \pi(Q) \; \bigr | \; \mu \left ( T^*_Q(Y^\vee) \right) \cap \SP(\caO^\vee) \text{ is nonempty}\right \}
 \]
are linearly independent elements of
$\bbZ_\st\Pi(\caO^\vee)$.
We need some additional notation.
For an object $\caS$ in $\caP(Y^\vee,K_i^\vee)$, write
\[
\CC_i(\caS) = a_1[\ol{T_{Q_1}^*(Y^\vee)}] + \cdots +a_r[\ol{T_{Q_r}^*(Y^\vee)}],
\]
and define
\[
\AV_\bbC(\caS) = G^\vee \cdot \bigcup_i 
\mu\left (\ol {T_{Q_i}^*(Y^\vee)} \right ) \; \subset \; \ol{\caO^\vee}.
\]
According to the irreducibility theorem of \cite{bb:i}, 
if $\caS$ is irreducible, then $\AV_\bbC(\caS)$ is the closure of
a single adjoint orbit.  The results of \cite{bv:i,bv:ii} show that
the orbit must be special.  

Next recall the isomorphism
\[
\Phi \; : \; 
\bbZ_\st\Pi^{\lam(\caO^\vee)}(G/\bbR)
\lra 
\bigoplus_i \left ( \bbZ\caP(Y^\vee, K^\vee_i)/\ker(\CC_i) \right ) ^\star
\]
obtained from \eqref{e:stable} and \eqref{e:ccisom2}.
It follows from \cite[Theorem 27.12]{abv} that 
$\pi \in \bbZ_\st\Pi^{\lam(\caO^\vee)}(G/\bbR)$ belongs to the subspace
$\bbZ_\st\Pi(\caO^\vee)$ if and only if there is an irreducible object
$\caS$ in some $\caP(Y^\vee, K^\vee_i)$ 
with $\AV_\bbC(\caS) = \ol{\caO^\vee}$ such that $\Phi(\pi)(\caS) \neq 0$.
As a consequence, suppose 
$Q$ is an orbit of $K_i^\vee$ on $Y^\vee$ such that
\[
\mu\left ( \ol {T_Q^*(Y^\vee)} \right ) \cap \SP(\caO^\vee) \neq \emptyset.
\]
Set
\[
m_Q' = m_Q \circ \CC_i
\in \left (\bbZ\caP(Y^\vee, K^\vee_i)/\ker(\CC_i)\right )^\star.
\]
There exists at least one irreducible object $\caS$ in $\caP(Y^\vee, K^\vee_i)$
whose support is the closure of $Q$.  (For example, let $\caS$  be the DGM
extension of the trivial local system on $Q$.  The
fundamental class of the closure of $T_Q^*(Y^\vee)$ appears with multiplicity
one in $\CC_i(\caS)$.)
Since $\mu \left ( \ol {T_Q^*(Y^\vee)} \right ) \cap \SP(\caO^\vee)$ is nonempty
by hypothesis, and
since $\AV_\bbC(\caS)$ must be the closure of a special orbit, it follows
that $\AV_\bbC(\caS)= \ol{\caO^\vee}$.  Therefore, by the discussion above,
the element $\pi(Q) \in \bbZ_\st\Pi^{\lam(\caO^\vee)}(G/\bbR)$ corresponding
to $m_Q$ is  nonzero and belongs to $\bbZ_\st\Pi(\caO^\vee)$.  In other
words,
\[
 \left \{ \pi(Q) \; \bigr | \; \mu \left ( T^*_Q(Y^\vee) \right) \cap \SP(\caO^\vee)
\textrm{ is nonempty} \right \}
\; \subset \; \bbZ_\st\Pi(\caO^\vee).
\]
Because the $m_Q$ are clearly linearly independent, so are the elements
$\pi(Q)$ on the left-hand side above.

It remains to show the elements in \eqref{e:stablefunctionals2} are indeed a basis of
$\bbZ_\st\Pi(\caO^\vee)$.  Because of the linear independence just established, it suffices to check
\begin{equation}
\label{e:ineq}
\sum_i \# \left \{ Q \in K_i^\vee \bs Y^\vee \; \bigr | \; \mu \left ( T^*_Q(Y^\vee) \right) \cap \SP(\caO^\vee)
\right \} \geq \dim_\bbZ\bbZ_\st\Pi(\caO^\vee).
\end{equation}
The following general result  will be the main tool we use for counting the left-hand side.
\begin{prop}
\label{p:geomcount}
As above, assume $\caO^\vee$ is even (but do not necessarily assume the Spaltenstein 
dual of $\caO^\vee$ is even).
With notation as in \eqref{e:unipic}--\eqref{e:unipY}, fix a symmetric subgroup
$K^\vee$ of $G^\vee$ and write $\frg^\vee = \frk^\vee \oplus \frs^\vee$ for
the corresponding Cartan decomposition.
Fix a nilpotent $K^\vee$ orbit $\caO^\vee_{K}$ on $\frs^\vee$.  
Let $c(\caO^\vee_{K})$ denote the number of $K^\vee$ orbits 
$Q$ on $Y^\vee$
such that $\mu(T^*_{Q}Y^\vee)$ meets $\caO^\vee_{K}$ 
in a dense open set.  Then,
for $x^\vee \in \caO^\vee_{K}$,
\[
c(\caO^\vee_{K}) = \dim\Hom_{W(\frl^\vee(\caO^\vee))}\left( \sgn, \Sp(x)^{A_{K^\vee}(x^\vee)}\right)
\]
with notation for the Springer correspondence in Proposition \ref{p:rossmann}.
\end{prop}

\noindent {\bf Proof.} 
This is a general result (and doesn't have anything to do with the dual group).
It follows from Rossmann's theory applied to the partial flag setting.
See \cite[Section 2]{cnt} for a proof.
\qed

\medskip

\begin{prop}
\label{p:geomcount2}
In the setting of Proposition \ref{p:geomcount}, assume further that the
Spaltenstein dual of $\caO^\vee$ is even, and
\[
G^\vee \cdot \caO^\vee_K \subset \SP(\caO^\vee).
\]
Then numbers $c(\caO^\vee_{K^\vee})$ appearing in Proposition \ref{p:geomcount}
are all nonzero.  More precisely,
\[
\dim\Hom_{W(\frl^\vee(\caO^\vee))}\left( \sgn, \Sp(x^\vee)^{A_{G^\vee}(x^\vee)}\right) =1,
\]
and since $A_{K^\vee}(\xi) \rightarrow A_{\vG}(\xi)$,
\begin{equation}
\label{e:orbitcount}
\dim\Hom_{W(\frl^\vee(\caO^\vee))}\left( \sgn, \Sp(x^\vee)^{A_{K^\vee}(x^\vee)}\right) \geq 1.
\end{equation}
\end{prop}

\noindent{\bf Proof.}  Section 5 of \cite{bv} shows that 
\[
\dim\Hom_{W(\frl^\vee(\caO^\vee))}\left( \sgn, U \right) =1
\]
for an irreducible representation $U$ in the left cell representation $V(\caO^\vee)$.
So the current proposition follows from Proposition \ref{p:cellcomp2}.
\qed

\bigskip

We now return to \eqref{e:ineq}.  By Proposition \ref{p:geomcount2},
\begin{equation}
\label{e:orbitineq}
\sum_i \# \left \{ Q \in K_i^\vee \bs Y^\vee \; \bigr | \; \mu \left ( T^*_Q(Y^\vee) \right) \cap \SP(\caO^\vee)
\right \} \geq 
\sum_i \#\left ( K_i^\vee \backslash (\SP(\caO^\vee) \cap \frs^\vee_i) \right ).
\end{equation}
By \eqref{e:maineq} (which was proved in the previous section), the right-hand side equals
$\dim_\bbZ\bbZ_\st\Pi(\caO^\vee)$.  This proves \eqref{e:ineq}, and hence completes
the proof of Proposition \ref{p:functionals}. 
\end{proof}

\begin{cor}
\label{c:Qcan}
In the setting of Proposition \ref{p:geomcount2}, there exists a unique orbit $Q = Q(\caO^\vee_K)$ of $K^\vee$ on
$Y^\vee$ such that
\[
\mu\left ( \ol{T_{Q}^*(Y^\vee)} \right ) = \ol{\caO_K^\vee}.
\]
\end{cor}
\begin{proof}
The proof of Proposition \ref{p:functionals} shows that the inequality
in \eqref{e:orbitineq} must be an equality.  Hence the inequality
in \eqref{e:orbitcount} in an equality.  Hence the number
$c(\caO^\vee_K)$ in Proposition \ref{p:geomcount} must be 1.  This
proves the corollary.
\end{proof}

{\noindent \bf Proof of Theorem \ref{t:main}.}
In the setting of Theorem \ref{t:main}, given $\caO_K^\vee \in K^\vee_i \bs (\SP(\caO^\vee) \cap \frs_i^\vee)$
define $Q(\caO^\vee_K)$ as in Corollary \ref{c:Qcan}.  In the notation of Proposition 
\ref{p:functionals} set 
\begin{equation}
\label{e:extradef}
\pi(\caO^\vee_K) := \pi(Q(\caO^\vee_K)).
\end{equation}
Then by Proposition \ref{p:functionals}(2),
\[
\left \{ \pi(\caO^\vee_K) \; \bigr | \; \caO^\vee_K \in \bigcup_i K^\vee_i \bs (\SP(\caO^\vee) \cap \frs^\vee_i)
\right \}
\]
is a basis for $\bbZ_\st\Pi(\caO^\vee)$.
\qed

\bibliographystyle{amsalpha}

\end{document}